\documentclass[11pt]{article}

\usepackage{booktabs}
\usepackage{fullpage}

\usepackage[round]{natbib}

\usepackage{amsfonts}
\usepackage{amsmath}
\usepackage{amssymb}
\usepackage[colorlinks = true, pdfstartview = FitV, linkcolor = blue, citecolor = blue, urlcolor = blue]{hyperref}
\usepackage{enumerate}
\usepackage{enumitem}

\usepackage{graphicx}

\usepackage[framed, amsthm, hyperref]{ntheorem}
\usepackage[framemethod=TikZ]{mdframed}

\usepackage[capitalise]{cleveref}
\crefname{equation}{}{}
\crefname{figure}{Figure}{Figures}
\creflabelformat{equation}{\textup{(#2#1#3)}}
\crefname{assumption}{Assumption}{Assumptions}
\crefname{condition}{Condition}{Conditions}

\input{formatting/newcommands.tex}
\newmdtheoremenv[%
	linewidth = 1pt,%
	roundcorner = 10pt,%
	leftmargin = 0,%
	rightmargin = 0,%
	backgroundcolor = white,%
	outerlinecolor = black,%
	splittopskip = \topskip,%
	ntheorem = true,%
]{theorem}{Theorem}

\newmdtheoremenv[%
	linewidth = 1pt,%
	roundcorner = 10pt,%
	leftmargin = 0,%
	rightmargin = 0,%
	backgroundcolor = white,%
	outerlinecolor = black,%
	splittopskip = \topskip,%
	ntheorem = true,%
]{corollary}{Corollary}

\newmdtheoremenv[%
	linewidth = 1pt,%
	roundcorner = 10pt,%
	leftmargin = 0,%
	rightmargin = 0,%
	backgroundcolor = green!3,%
	outerlinecolor = blue!70!black,%
	splittopskip = \topskip,%
	ntheorem = true,%
]{lemma}{Lemma}

\newmdtheoremenv[%
	linewidth = 1pt,%
	roundcorner = 10pt,%
	leftmargin = 0,%
	rightmargin = 0,%
	backgroundcolor = blue!3,%
	outerlinecolor = blue!70!black,%
	splittopskip = \topskip,%
	ntheorem = true,%
]{definition}{Definition}

\newmdtheoremenv[%
	linewidth = 1pt,%
	roundcorner = 10pt,%
	leftmargin = 0,%
	rightmargin = 0,%
	backgroundcolor = green!3,%
	outerlinecolor = blue!70!black,%
	splittopskip = \topskip,%
	ntheorem = true,%
]{proposition}{Proposition}

\newmdtheoremenv[%
	linewidth = 1pt,%
	roundcorner = 10pt,%
	leftmargin = 0,%
	rightmargin = 0,%
	backgroundcolor = green!3,%
	outerlinecolor = blue!70!black,%
	splittopskip = \topskip,%
	ntheorem = true,%
]{condition}{Condition}

\newmdtheoremenv[%
	linewidth = 1pt,%
	roundcorner = 10pt,%
	leftmargin = 0,%
	rightmargin = 0,%
	backgroundcolor = white,%
	outerlinecolor = black,%
	splittopskip = \topskip,%
	ntheorem = true,%
]{assumption}{Assumption}

\theoremstyle{definition}
\newmdtheoremenv[%
linewidth = 1pt,%
roundcorner = 10pt,%
leftmargin = 0,%
rightmargin = 0,%
backgroundcolor = cyan!3,%
outerlinecolor = blue!70!black,%
splittopskip = \topskip,%
ntheorem = true,%
]{example}{Example}

\theoremstyle{definition}
\newmdtheoremenv[%
	linewidth = 1pt,%
	roundcorner = 10pt,%
	leftmargin = 0,%
	rightmargin = 0,%
	backgroundcolor = red!3,%
	outerlinecolor = blue!70!black,%
	splittopskip = \topskip,%
	ntheorem = true,%
]{remark}{Remark}

\theoremstyle{nonumberplain}
\newmdtheoremenv[%
	linewidth = 1pt,%
	roundcorner = 10pt,%
	leftmargin = 0,%
	rightmargin = 0,%
	backgroundcolor = white,%
	outerlinecolor = black,%
	splittopskip = \topskip,%
	ntheorem = true,%
]{informal}{Main Result}

\renewcommand\qedsymbol{$\blacksquare$}
\setlist[enumerate,1]{%
	leftmargin=*, wide=0em, noitemsep, nolistsep, label = {\bfseries \arabic*.}
}
\setlist[itemize,1]{%
	leftmargin=*, wide=0em, noitemsep, nolistsep
}

\begin{document}

\title{Invexifying Regularization of Non-Linear Least-Squares Problems}
\author{
    Rixon Crane\footnote{School of Mathematics and Physics, University of Queensland, Australia. Email: r.crane@uq.edu.au} 
    \qquad 
    Fred Roosta\footnote{School of Mathematics and Physics, University of Queensland, Australia, and International Computer Science Institute, Berkeley, USA. Email: fred.roosta@uq.edu.au}
}
\maketitle

\begin{abstract}
We consider regularization of non-convex optimization problems involving a non-linear least-squares objective. By adding an auxiliary set of variables, we introduce a novel regularization framework whose corresponding objective function is not only provably invex, but it also satisfies the highly desirable Polyak--Lojasiewicz inequality for any choice of the regularization parameter. Although our novel framework is entirely different from the classical $\ell_2$-regularization, an interesting connection is established for the special case of under-determined linear least-squares. In particular, we show that gradient descent applied to our novel regularized formulation converges to the same solution as the linear ridge-regression problem. Numerical experiments corroborate our theoretical results and demonstrate the method’s performance in practical situations as compared to the typical $\ell_2$-regularization.
\end{abstract}

\section{Introduction}\label{section: introduction}

Consider the non-linear regression problem
\begin{align}\label{eq: fx}
	\min_{\bfx}\:\Bigl\{\fx\triangleq\hf\bigl\|\gx\bigr\|^2\Bigr\},
\end{align}
where $ \bfg : \Rd \rightarrow \Rn $ is a non-linear sufficiently smooth mapping.
Problems of this form are very common in the field of machine learning, where $ \bfg $ and $ \bfx $, respectively, represent the model to train/fit and its parameters.
For example, a nonlinear regression task involving a neural-network model 
$ \bfg $ with weights parameters $ \bfx $ and a least-squares loss 
\citep{goodfellow2016deep}. Here, each component of the vector valued function $ \bfg $ can correspond to an individual input from a training dataset of size $ n $.
Beyond machine learning, such problems arise frequently across many areas of 
science and engineering, e.g., PDE-constrained inverse problems 
\citep{rodoas1,rodoas2}. 
Historically, low-dimensional models ($ d\leq n $) have been used 
to great effect over a diverse range of practical settings 
\citep{bates2007nonlinear}.
In recent years, driven by modern machine-learning applications, more attention 
has been given to problems \eqref{eq: fx} in which $ f $ is high-dimensional 
($ d \gg 1 $) and even over-parameterized ($ d \gg n $).
Notable examples include generative models like auto-encoders, 
\citep{Chen_2019_ICCV}, and generative adversarial networks, 
\citep{Mao_2017_ICCV}. In all these cases, the training/fitting often involves some form of regularization of \cref{eq: fx}. In principle, there are two, often complementary, perspectives on regularization: one from the viewpoint of improving generalization and predictive performance, and the other with regards to improving the optimization landscape as a way to facilitate the training procedure.

\paragraph{Regularization to Improve Generalization.}
In all machine learning problems, avoiding the pitfalls of over-fitting and improving generalization performance constitute a major challenge in the training procedure \citep{mohri2018foundations,shalev2014understanding}. This is particularly important in over-parameterized settings or small-data learning tasks, where over-fitting can be a major hindrance in obtaining good out-of-sample predictive performance. 

In this light, various regularization techniques have been proposed to mitigate over-fitting in a variety of settings. These techniques range from the traditional ridge-type $\ell_{2}$-regularization, to more recent  techniques such as weight decay \citep{krogh1992simple,loshchilov2017decoupled}, dropout \citep{hinton2012improving,baldi2013understanding}, and various data augmentation techniques \citep{shorten2019survey}. One can broadly, and perhaps loosely, categorize these regularization techniques into two main groups: those that directly regularize the model to reduce its representational capacity, e.g.,  $\ell_{2}$-regularization, weight decay, and dropout, and those that instead aim to enhance the size and quality of training datasets, e.g., data augmentation. It is also typical to see these regularization techniques used in conjunction with one another.

\paragraph{Regularization to Facilitate Optimization.}
Unless $ \bfg $ is a linear map, the problem \cref{eq: fx} 
amounts to a \emph{non-convex} optimization problem.
In this light, the vast majority of optimization research has typically focused on 
developing (general purpose) optimization algorithms that, in the face of such 
non-convexity, come equipped with strong convergence guarantees, e.g., 
\cite{NEURIPS2019_b8002139, NEURIPS2019_d202ed5b, NEURIPS2019_d9fbed9d,
       NEURIPS2019_4d0b954f, pmlr-v97-yu19c,       pmlr-v97-haddadpour19a,xuNonconvexTheoretical2017, tripuraneni2017stochasticcubic,
       bellavia2020adaptive,       wang2019stochastic,
       blanchet2019convergence,    gupta2018shampoo,
       anil2020second,	yao2018inexact,	liu2019stability}.
However, most of these methods involve subtleties and disadvantages that can
make their use far less straightforward in many training procedures.
For example, complex, computationally intensive, and non-trivial steps in the 
algorithm and/or difficulty in fine-tuning the underlying hyper-parameters.

As a result, to improve the optimization landscape of \cref{eq: fx} and lessen the challenges faced by the optimization procedure, and hence to facilitate the training procedure, one can ``perturb'' the original problem 
\eqref{eq: fx} to a nearby problem, which exhibits more favorable structural properties, e.g., better condition number, ``rounder'' level-sets, better smoothness properties, etc. 
Perhaps, the simplest and the most well-known strategy for this purpose is the traditional $ \ell_2 $-regularization, which instead of \cref{eq: fx} considers the alternative objective $ \fx + \lambda \| \bfx \|^{2} $ for 
some constant $ \lambda > 0 $.
This technique attempts to improve upon the conditioning of the problem \cref{eq: fx} by ``smoothing out'' the highly non-convex regions of its landscape.

\paragraph{Complementary in Spirit, but Conflicting in Reality.}
Although, these two view-points on regularization are complementary in spirit, the underlying techniques employed can at times be at odds with each other. For example, to smooth out the optimization landscape, consider the $\ell_{2}$ regularization of \cref{eq: fx} as $ \fx + \lambda \| \bfx \|^{2} $. Although this new objective might have better structural properties than the original non-regularized problem, e.g., better conditioning and ``rounder'' level-sets, it can still be non-convex and challenging to optimize. In fact, it will only be convex for sufficiently large values of $\lambda$, and the lower-bound of such values relies on practically unknowable constants related to the spectrum of the Hessian of $ f $.
However, for such large values, the benefits offered by the convexity of the regularized function comes at a great cost: the obtained solution will most likely be far too biased to be relevant to the original
problem or any out-of-sample generalizations \citep{golatkar2019time}. In this light, a regularization technique that can offer the best of both worlds is highly desirable. Our aim in this paper is to propose one such regularization method that not only can greatly improve generalization performance, but it also provably offers a structurally easier model to train/fit.

\subsection{Our Approach and Contributions}
Our novel regularization framework involves adding an auxiliary set of variables, and in essence ``lifting'' the original optimization problem to higher dimensions. The additional set of variables are directly coupled with the output of the non-linear mapping $ \bfg $. 
In this light, not only does our framework regularize the model itself, but it also non-trivially and adaptively interacts with the input training data throughout training. 
We show that the new regularized problem, though still non-convex, will enjoy several highly-desirable properties, which can greatly facilitate the training procedure.
Remarkably, these properties hold regardless of the non-linear mapping,  $ \bfg $, and the regularization parameter.

Assuming $ \bfg $ is differentiable, we recall that analyzing the Jacobian of $ \bfg $, denoted by $ \Jg $, 
gives insights to the optimization landscape of \eqref{eq: fx}, e.g.,
\citet{zhang2018three,golatkar2019time,li2018visualizing}.
In general, this Jacobian may be rank-deficient over sets of points in $ \Rd $.
In such situations, the optimization landscape may exhibit a high degree of
non-convexity.
The key insight is that when $ \Jg $ has full row-rank, the structure of 
\eqref{eq: fx} guarantees valuable properties of $ f $, which we shall soon
detail.
To enforce this situation, we introduce an auxiliary variable $ \bfp\in\Rn $
and consider the new optimization problem
\begin{align}\label{eq: fhxp}
	\min_{\bfx,\bfp}\:
	\Bigl\{\fhxp\triangleq\hf\bigl\|\gx+\lambda\bfp\bigr\|^2\Bigr\},
\end{align}
where $ \lambda>0 $ is a pre-selected constant.
Clearly, the the Jacobian of $ \gx+\lambda\bfp $, 
which is $ \begin{bmatrix} \Jgx & \lambda\bfI \end{bmatrix} $, 
is guaranteed to have full row-rank for any value of $ \lambda > 0 $. 
This will have profound, and highly desirable, implications on the loss landscape of $ \fh $. 
Specifically, we show that \eqref{eq: fhxp} induces the following properties, which are 
formalized in Section~\ref{section: theory}.

\paragraph{Contributions.}  Let us briefly highlight our contributions.
\begin{enumerate}[label = {\bfseries (\roman*)}]
	\smallskip \item 
	We first show that our novel regularization framework coincides, 
	in an intriguing manner, with the classical $\ell_{2}$-regularization, 
	for the special case when $ \bfg $ is an affine mapping.
	
	\smallskip \item 
	For arbitrarily non-linear $ \bfg $, we then show that the function $ \fh $, 
	in \eqref{eq: fhxp}, is \emph{invex for any value of $ \lambda $}	
	(Theorem~\ref{theorem: invex-pl}-\ref{theorem: invex}), 
	which is a pleasant ``middle ground'' between convexity and non-convexity. 
	
	\smallskip \item 
	More importantly, we show the function  $ \fh $ satisfies the highly desirable
	\emph{Polyak--{\L}ojasiewicz (PL) inequality for any value of $ \lambda $} 
	(Theorem~\ref{theorem: invex-pl}-\ref{theorem: pl}), which allows for exponentially 
	fast convergence of many optimization algorithms.
	
	\smallskip \item We finally study the empirical performance of our novel regularization framework on several challenging ML problems.
\end{enumerate}

\newpage

\begin{remark}
Note that $\min_{\bfx,\bfp}\; \fhxp = 0$, i.e., the new regularization amounts to an interpolating model.
In sharp contrast to the classical $ \ell_{2} $-regularization whose optimal 
value is almost always non-trivial, the interpolation property of 
\eqref{eq: fhxp} provides a significantly useful feature in practice.
Namely, it allows for monitoring convergence by simply inspecting the training 
loss. 
\end{remark}

To our knowledge, the reformulation \eqref{eq: fhxp} and the addition of 
auxiliary variable $ \bfp $ is novel.
However, our approach here, in some sense, can be loosely connected with those 
presented by \citet{liang2018adding,kawaguchi2020elimination} in which by 
adding one special neuron per output unit, the loss landscape is modified in a 
way that all sub-optimal local minima of the original problem are eliminated, 
i.e., one can recover the global optima of the original problem from the local 
minima of the modified problem.
However, although the approach by \citet{kawaguchi2020elimination} applies more 
generally beyond non-linear least squares, the loss landscape remains highly 
non-convex with potentially many local minima, saddle points, and local maxima. 
Whereas, the invexity of \eqref{eq: fhxp} implies that 
all its stationary points are global optima.

One may also see some similarities between our approach and that of 
Hamiltonian Monte Carlo methods (HMC), where the original sample space 
involving ``position'' variables, is lifted to include auxiliary ``momentum'' 
variables \citep{neal2011mcmc,betancourt2017conceptual}.
In this light, the sampling from the position-momentum phase-space in HMC 
resembles optimization of $ \fh $ over the augmented space 
$ (\bfx,\bfp) \in \real^{d + n} $.

\section{Theoretical Analysis}\label{section: theory}

Comparing \eqref{eq: fhxp} to the usual $ \ell_2 $-regularization, one could 
suggest that in the reformulation \eqref{eq: fhxp}, the data is being regularized instead of the parameters. 
This introduces highly non-trivial and adaptive interactions between the data and the variable $ \bfp $ throughout the training procedure. 
In this section, we aim to develop some theoretical insights into structural properties of 
the proposed regularization method \cref{eq: fhxp} and on the consequential effects to the training 
procedure. 

\subsection{Connection to \texorpdfstring{$ \ell_{2} $}{l2}-regularization}
To study properties of \cref{eq: fhxp}, it might be more insightful to start with the simplest possible case where $ \bfg $ is just an affine map. Clearly, this setting amounts to \cref{eq: fx} being simply an ordinary least-squares problem. 
Although the reformulation \cref{eq: fhxp} is completely different than the classical $\ell_{2}$-regularization, surprisingly, it turns out that when $ \bfg $ is affine these two regularization methods coincide. More specifically, in \cref{theorem: lls}, we show that gradient descent (GD) applied to \cref{eq: fhxp} converges to the unique solution of the classical ridge-regression problem.

\newpage
\begin{theorem}[Linear Least-Squares]\label{theorem: lls}
	Consider \eqref{eq: fx} with $ \gx = \bfA\bfx - \bfb $,
	such that $ \bfA \in \Rnd $ has full row-rank and  $ \bfb \in \Rn $,
	and then consider applying GD to \eqref{eq: fhxp},
	using a sufficiently small fixed step-size
	and starting from the origin $ (\xo,\po) = (\zero,\zero) $.
	We have
	\begin{enumerate}[label = {\bfseries (\roman*)}]
		\item The iterates $ \xt $ converge to the unique solution of the ridge-regression problem.
		Namely,
		\begin{align*}
			\lim_{t\rightarrow\infty} \xt
			=
			\bfx^{*}(\lambda)
			\triangleq
			\arg\min_{\bfx}
			\bigl\{ \| \bfA\bfx - \bfb \|^2 + \lambda \| \bfx \|^2 \bigr\}
			=
			( \bfA^{\transpose}\bfA + \lambda^2\bfI )^{-1} \bfA^{\transpose}\bfb.
		\end{align*}
		
		\item Furthermore, $ \bfx^{*}(\lambda) $ behaves consistently with
		the ridge-regression solution in the limiting cases of $ \lambda $.
		That is,
		\begin{align*}
			\lim_{\lambda\rightarrow0} \bfx^{*}(\lambda) =\bfA^{\dagger}\bfb,
			\:\:
			\lim_{\lambda\rightarrow\infty} \bfx^{*}(\lambda) = \zero,
			\:\:
			\lim_{\lambda\rightarrow0} f\bigl( \bfx^{*}(\lambda) \bigr) = 0,
			\:\:
			\lim_{\lambda\rightarrow\infty} f\bigl( \bfx^{*}(\lambda) \bigr) 
			= \| \bfb \|^2,
		\end{align*}
		where $ \bfA^{\dagger} $ denotes the Moore–Penrose inverse of $ \bfA $
		and $ f $ is as in \eqref{eq: fx}.
	\end{enumerate}
\end{theorem}

\begin{proof}
    Here, our regularized function \eqref{eq: fhxp} is given by
    \begin{align*}
        \fhxp
        &=
        \hf\|\bfA\bfx-\bfb+\lambda\bfp\|^2
        =
        \hf
        \biggl\|
        \begin{bmatrix} \bfA & \lambda\bfI \end{bmatrix}
        \begin{bmatrix} \bfx \\ \bfp \end{bmatrix}
        - \bfb
        \biggr\|^{2}.
    \end{align*}
    This involves a full row-rank coefficient matrix.
    Therefore, using a sufficiently small fixed step-size and starting from the 
    origin $ (\xo,\po) = (\zero,\zero) $, it is a well known result that 
    the iterates from GD applied to this function will converge to the 
    minimum-norm solution; for example, refer to \cite{JMLR:v19:18-188}.
    The minimum-norm solution is
    \begin{align*}
        \begin{bmatrix} \bfx^{*}(\lambda) \\ \bfp^{*}(\lambda) \end{bmatrix}
        =
        \begin{bmatrix} \bfA & \lambda\bfI \end{bmatrix}^{\dagger}\bfb
        =
        \begin{bmatrix} \bfA^{\transpose} \\ \lambda\bfI \end{bmatrix}
        ( \bfA\bfA^{\transpose} + \lambda^2\bfI )^{-1} \bfb
        =
        \begin{bmatrix}
            ( \bfA^{\transpose}\bfA + \lambda^2\bfI )^{-1} \bfA^{\transpose}\bfb \\
            \lambda ( \bfA\bfA^{\transpose} + \lambda^2\bfI )^{-1} \bfb
        \end{bmatrix}.
    \end{align*}
    Clearly, 
    $ \lim_{\lambda\rightarrow0} \bfx^{*}(\lambda) = \bfA^{\dagger}\bfb $ and 
    $ \lim_{\lambda\rightarrow\infty} \bfx^{*}(\lambda) = \zero $.

    Since GD will converge to a global optimum of $ \fh $, we have
    $\bfp^{*}(\lambda) = \bigl( \bfb - \bfA\bfx^{*}(\lambda) \bigr) / \lambda $.
    Together with the property that all iterates $ \xt $ are in the range of 
    $ \bfA^{\transpose} $ by definition of GD, we have
    \begin{align*}
        \bfx^{*}(\lambda)
        =
        \bfA^{\transpose} ( \bfA\bfA^{\transpose} )^{-1}
        \bigl( \bfb - \lambda\bfp^{*}(\lambda) \bigr).
    \end{align*}
    Therefore,
    \begin{align*}
        f\bigl( \bfx^{*}(\lambda) \bigr)
        =
        \hf
        \Bigl\|
            \bfA\bfA^{\transpose} ( \bfA\bfA^{\transpose} )^{-1}
            \bigl( \bfb - \lambda\bfp^{*}(\lambda) \bigr)
            - \bfb
        \Bigr\|^{2}
        =
        \hf \bigl\| \lambda\bfp^{*}(\lambda) \bigr\|^{2}
        =
        \hf 
        \bigl\| 
            \lambda^{2} ( \bfA\bfA^{\transpose} + \lambda^2\bfI )^{-1} \bfb 
        \bigr\|^{2},
    \end{align*}
    with
    $ \lim_{\lambda\rightarrow0} f\bigl( \bfx^{*}(\lambda) \bigr) = 0 $,
    $ \lim_{\lambda\rightarrow\infty} f\bigl( \bfx^{*}(\lambda) \bigr) 
    = \| \bfb \|^2 $.
    \qedsymbol
\end{proof}

\subsection{General Non-linear Mapping \texorpdfstring{$ \bfg $}{g}}
As mentioned in Section~\ref{section: introduction}, a fundamental distinction 
between our proposed formulation and $ \ell_{2}$-regularization becomes apparent 
when moving beyond the simple linear least-squares case of 
Theorem~\ref{theorem: lls}.
This difference mainly lies in the fact that our new objective function retains fundamental aspects 
of its structure, \emph{regardless} of the choice of $ \lambda > 0 $ or the function $ \bfg $.
In particular, the optimization landscape of the typical 
$ \ell_{2} $-regularized function, given by $ \fx + \lambda \|\bfx\|^2 $, 
can range from convex to highly non-convex, all depending on the non-trivial, 
and often unknown, interplay between curvature of $ f $ and the choice of 
$ \lambda $. 
Whereas, the optimization landscape induced from our regularization in \eqref{eq: fhxp} retains its highly desirable structural properties of invexity and PL property, for any $\lambda > 0$.

Invexity was introduced to extend the sufficiency of the first-order optimality condition beyond simple convex programming \citep{mishra2008invexity,cambini2008generalized}.
As a result, a differentiable function, for example, is invex if and only if 
all its critical points are global minima.
PL inequality \citep{10.1007/978-3-319-46128-1_50}, in fact, characterizes a special class of invex functions for which many optimization algorithms can be shown to converge exponentially fast.
In its introductory paper, \citet{Polyak1963} showed that GD enjoys a global 
linear convergence-rate under this condition.
In recent years, it has garnered increased attention from the machine-learning 
community.
For example, it has been at the heart of many convergence-proofs outside the
limitations of strong-convexity, e.g.,
\citep{10.1007/978-3-319-46128-1_50,bassily2018exponential,pmlr-v89-vaswani19a,gower2021sgd,yuan2018stagewise,ajalloeian2020analysis}.
This inequality also has connections to other topics of interest, including over-parametrization and interpolation \citep{pmlr-v89-vaswani19a}.
For example, it has been shown that sufficiently-wide and over-parameterized 
neural networks, under certain assumptions, induce the PL inequality 
\citep{liu2020toward}.

\begin{theorem}[Invexity and PL Inequality]\label{theorem: invex-pl}
	Suppose $ \bfg $ is differentiable and let $ \lambda > 0 $. The function $ \fh $ in \eqref{eq: fhxp} has the following properties.
	\begin{enumerate}[label = {\bfseries (\roman*)}]
		\item \label{theorem: invex} \textbf{(Invexity)} All stationary points are global minima.
		In particular, for all $ \bfx,\bfy \in \Rd $ and $ \bfp,\bfq \in \Rn $, we have
		\begin{align*}
			\fhxp - \fhyq
			&\geq
			\bigl\langle \etaxpyq, \nabla\fhyq \bigr\rangle,
		\end{align*}
		where 
		\begin{align*}
			\etaxpyq =
		    \begin{bmatrix} \Jgy  \lambda\bfI \end{bmatrix}^{\dagger}
			\Bigl( 
			    \bigl( \gx + \lambda\bfp \bigr) 
			    - \bigl( \gy + \lambda\bfq \bigr) 
			\Bigr).
		\end{align*}
		
		\item \label{theorem: pl} \textbf{(Polyak--Łojasiewicz Inequality)}
		For all $ \bfx \in \Rd $ and $ \bfp \in \Rn $,  we have
		\begin{align*}
			\fhxp \leq \frac{1}{2\lambda^2} \bigl\| \nabla\fhxp \bigr\|^2.
		\end{align*}
	\end{enumerate}
\end{theorem}

\begin{proof}
    \hfill
	\begin{enumerate}[label = {\bfseries (\roman*)}]
	    \item \textbf{(Invexity)} 
	    Consider the Jacobian of $ \gx + \lambda \bfp $, which is
        $ \begin{bmatrix} \Jgx & \lambda\bfI \end{bmatrix} $.
        As it is guaranteed to have full row-rank,
        its pseudo-inverse acts as a right inverse.
        Together with the convexity of the $\ell_{2}$-norm squared, 
        we have, for all $ \bfx,\bfy \in \Rd $ and $ \bfp,\bfq \in \Rn $,
        \begin{align*}
            \fhxp - \fhyq
            &=
            \hf \bigl\| \gx + \lambda \bfp \bigr\|^{2}
            - \hf \bigl\| \gy + \lambda \bfq \bigr\|^{2}
            \\ &\geq
            \Bigl\langle
            \bigl( \gx + \lambda\bfp \bigr) - \bigl( \gy + \lambda\bfq \bigr),
            \bigl( \gy + \lambda\bfq \bigr)
            \Bigr\rangle
            \\ &=
            \bigl\langle \etaxpyq, \nabla\fhyq \bigr\rangle,
        \end{align*}

	    \item \textbf{(Polyak--Łojasiewicz Inequality)}
	    Replacing $ \bfI $ with
	    $ 
	    \Bigl(
	        \begin{bmatrix} \Jgx & \lambda\bfI \end{bmatrix}
	        \begin{bmatrix} \Jgx & \lambda\bfI \end{bmatrix}^{\dagger} 
	    \Bigr)^{\transpose}
	    $, we have
	    \begin{align*}
	        \fhxp 
	        &=
	        \hf \Bigl\| \bfI \bigl( \gx + \lambda \bfp \bigr) \Bigr\|^{2}
	        \\ &=
	        \hf
	        \biggl\|
	        \Bigl(
    	        \begin{bmatrix} \Jgx & \lambda\bfI \end{bmatrix}^{\dagger} 
    	    \Bigr)^{\transpose}
    	    \begin{bmatrix} \Jtgx \\ \lambda \bfI \end{bmatrix} 
    	    \bigl( \gx + \lambda \bfp \bigr)
    	    \biggr\|^{2}
    	    \\ &=
	        \hf
	        \biggl\|
	        \Bigl(
    	        \begin{bmatrix} \Jgx & \lambda\bfI \end{bmatrix}^{\dagger} 
    	    \Bigr)^{\transpose}
    	    \nabla \fhxp
    	    \biggr\|^{2}.
	    \end{align*}
	    Expanding and using that fact that the
	    smallest non-zero eigenvalue of
	    $ 
        \begin{bmatrix} \Jgx & \lambda\bfI \end{bmatrix}^{\transpose}
        \begin{bmatrix} \Jgx & \lambda\bfI \end{bmatrix}
	    $
	    is at least $ \lambda^{2} $, we obtain the PL inequality.     \qedsymbol
	\end{enumerate}
\end{proof}

So far, we have only made the basic assumption of $ \bfg $ being differentiable.
Yet, from this alone, Theorem~\ref{theorem: invex-pl} shows that our regularized function \eqref{eq: fhxp} caries important structural guarantees, unlike that of $ \ell_2 $-regularization.
It is commonplace in machine-learning literature to have assumptions on 
smoothness, i.e., on the \emph{Lipschitz continuity} of the gradient,
to obtain additional convergence guarantees for a given optimization algorithm.
In this vein, by additionally making a reasonable assumption about twice 
continuous differentiability of $ \bfg $ on a compact ball, in 
Assumption~\ref{assumption: twice diff g}, our regularized objective 
\eqref{eq: fhxp} is shown to be also locally Lipschitz-continuous on a compact ball.
This is formalized in Theorem~\ref{theorem: lipschitz}.

\vspace{2mm}

\begin{assumption}\label{assumption: twice diff g}
    Given $ \lambda > 0 $ and $ \xo \in \Rd $, assume the non-linear mapping $ \bfg $ is  twice-continuously differentiable on the compact ball $ \sD(\xo,\lambda) = \left\{ \bfx \mid \| \bfx - \xo \| \leq R(\xo,\lambda) \right\} $, where we have defined $ R(\xo,\lambda) = 2\sqrt{2 \Delta_{0}} / \lambda $ and $ \Delta_{0} = \fxo - \min_{\bfx} \fx $.
\end{assumption}

\begin{theorem}[Smoothness of $ \fhxp $]\label{theorem: lipschitz}
	Under Assumption~\ref{assumption: twice diff g},
    let $ \widehat{\sD}(\xo,\lambda) $ denote the compact ball
    around $ (\xo, \zero) $ of radius $ R(\xo,\lambda) $, i.e.,
    \begin{align*}
        \widehat{\sD}(\xo,\lambda)
        =
        \Biggl\{
            (\bfx, \bfp)
            \mid
            \biggl\| \begin{bmatrix} \bfx - \xo \\ \bfp \end{bmatrix} \biggr\|
            \leq
            R(\xo,\lambda)
        \Biggr\},
    \end{align*}
    where $ R(\xo,\lambda) $ is as in \cref{assumption: twice diff g}.
    The function $ \fh $ in \eqref{eq: fhxp} is $ L(\xo,\lambda) $-smooth
    on $ \widehat{\sD} $ for some constant $ L(\xo,\lambda) $ that depends on 
    $ \xo $ and $ \lambda $.
\end{theorem}

\begin{proof}
    Consider the Hessian
    \begin{align*}
        \nabla^{2}\fhxp
        =
        \begin{bmatrix} \Jtgx \\ \lambda\bfI \end{bmatrix}
        \begin{bmatrix} \Jgx & \lambda\bfI \end{bmatrix}
        +
        \begin{bmatrix}
            \lambda \nabla^{2} \bigl\langle \bfp, \gx \bigr\rangle
            + \sum\bfg_{i}(\bfx)\nabla^{2}\bfg_{i}(\bfx)
            & \zero
            \\
            \zero & \zero
        \end{bmatrix}.
    \end{align*}
    Therefore, $ \bigl\| \nabla^{2}\fhxp \bigr\| $ is upper bounded by
    \begin{align*}
        \lambda^{2}
        + \bigl\| \Jgx \bigr\|^{2}
        + \bigl\| \nabla^{2}\gx \bigr\| \bigl\| \gx \bigr\|
        + \lambda \Bigl\| \nabla^{2} \bigl\langle \bfp, \gx \bigr\rangle \Bigl\|.
    \end{align*}
    \cref{assumption: twice diff g} implies
    there exist constants $ M(\xo,\lambda) < \infty$ and $N(\xo,\lambda) < \infty $,
    depending on $ \xo $ and $ \lambda $, such that $\bigl\| \gx - \gy \bigr\| \leq M(\xo,\lambda) \| \bfx - \bfy \|$ and $\bigl\| \Jgx - \Jgy \bigr\| \leq N(\xo,\lambda) \| \bfx - \bfy \|$, for all $ \bfx, \bfy \in \sD(\xo,\lambda) $ as defined in \cref{assumption: twice diff g}.
    Note that for any $ (\bfx, \bfp) \in \widehat{\sD}(\xo,\lambda) $, we have $ \bfx \in \sD(\xo,\lambda) $.
    Therefore,
    \begin{align*}
        \bigl\| \nabla^{2}\gx \bigr\|
        &\leq
        N(\xo,\lambda)
        \\
        \bigl\| \gx \bigr\|
        &\leq
        \bigl\| \gx - \gxo \bigr\| + \bigl\| \gxo \bigr\|
        \\ &\leq
        M(\xo,\lambda) \| \bfx - \xo \| + \bigl\| \gxo \bigr\|
        \\ &\leq
        M(\xo,\lambda) \cdot R(\xo,\lambda) + \bigl\| \gxo \bigr\|,
        \\
        \bigl\| \Jgx \bigr\|
        &\leq
        \bigl\| \Jgx - \Jgxo \bigr\| + \bigl\| \Jgxo \bigr\|
        \\ &\leq
        N(\xo,\lambda) \| \bfx - \xo \| + \bigl\| \Jgxo \bigr\|
        \\ &\leq
        N(\xo,\lambda) \cdot R(\xo,\lambda) + \bigl\| \Jgxo \bigr\|,
        \\
        \bigl\| \nabla^{2} \bigl\langle \bfp, \gx \bigr\rangle \bigl\|
        &\leq
        N(\xo,\lambda) \| \bfp \|
        \\ &\leq
        N(\xo,\lambda) \cdot R(\xo,\lambda),
    \end{align*}
    for all $ (\bfx, \bfp) \in \widehat{\sD} $.
    Substituting these into the upper bound of $ \bigl\| \nabla^{2}\fhxp \bigr\| $
    proves our claim.
    \qedsymbol
\end{proof}

Note that we have deliberately omitted the term $ \po $ from 
Assumption~\ref{assumption: twice diff g}.
Hence, our forthcoming convergence results assume that GD begins at the point
$ (\xo,\zero) $ instead of the more general $ (\xo,\po) $.
Although the generalized case is derivable without significant modifications, 
we find that it obfuscates the interoperability and significance of our results.

Since the introductory paper by \citet{Polyak1963}, it has become increasing 
well known in machine-learning literature that GD enjoys a global linear 
convergence-rate under the conditions of smoothness and the PL inequality.
For example, \citet{oymak2019overparameterized} considered that path taken by 
GD under these assumptions.
Corollary~\ref{corollary: nlls} is the immediate result of applying \cref{theorem: invex-pl,theorem: lipschitz} to their findings.

\begin{figure*}[h!tb]
    \centering
    \includegraphics{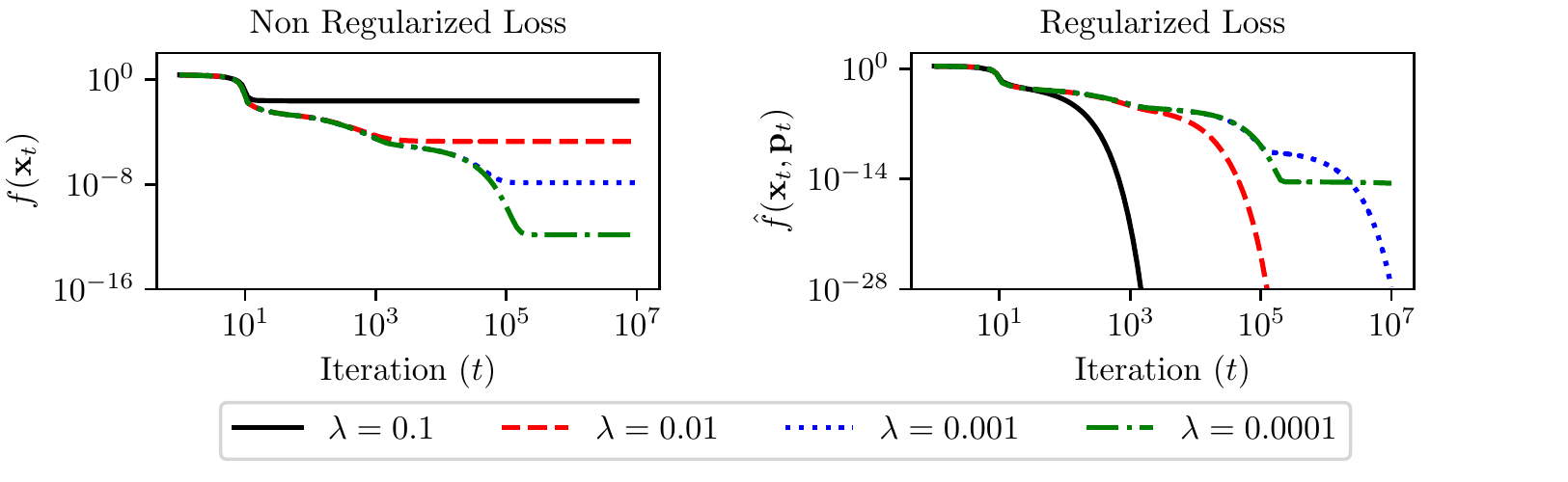}
    \caption{%
    The convergence of GD applied to our regularized function 
    \eqref{eq: fhxp} for several choices of $ \lambda $, on the problem of 
    binary classification with squared-loss.
    The non-regularized function value evaluated at the iterates is also plotted.
    We see that a smaller choice of $ \lambda $ relates to slower convergence in \eqref{eq: fhxp} and a lower non-regularized loss.
    Further study on this trade-off is left as future work.
    Note, the same step-size is used for all runs.
    }
    \label{figure: sigmoid}
\end{figure*}

\newpage
\begin{corollary}\label{corollary: nlls}
    Under \cref{assumption: twice diff g}, 
    let $ R(\xo,\lambda) $ and $ L(\xo,\lambda) $ be as in \cref{theorem: lipschitz}.
    Starting from the point $ (\xo,\po) = (\xo,\zero) $ and using a fixed step-size 
    $ \alpha \leq 1 / L(\xo,\lambda) $, the iterates $ (\xt,\pt) $ obtained via the
    GD updates
    \begin{align*}
        \begin{bmatrix} \xtt \\ \ptt \end{bmatrix}
        =
        \begin{bmatrix} \xt \\ \pt \end{bmatrix}
        -
        \alpha\nabla\fhxpt,
    \end{align*}
    satisfy, for all $ t \geq 0 $,
    \begin{align*}
        \fhxpt \leq \biggl( 1 - \frac{\lambda^{2}}{L(\xo,\lambda)} \biggr)^{t} \fh(\xo,\zero),  \quad \text{and} \quad
        \sum_{t=0}^{\infty}
        \biggl(
            \begin{bmatrix} \xtt \\ \ptt \end{bmatrix}
            -
            \begin{bmatrix} \xt \\ \pt \end{bmatrix}
        \biggr) \leq R(\xo,\lambda).
    \end{align*}
\end{corollary}

In \cref{corollary: nlls}, we have the global linear convergence-rate 
of GD on our regularized function \eqref{eq: fhxp}, for any $ \lambda > 0$.
Moreover, the total length of the path taken by the iterates of GD never 
exceeds $ R(\xo,\lambda) $, i.e., the iterates remain inside the ball $ \widehat{\sD}(\xo,\lambda) $.

Perhaps, one might also wonder how the solution $ \xt $ obtained from applying GD to the regularized problem \cref{eq: fhxp} relates to the non-regularized problem \cref{eq: fx}?
For iterates of GD, we have $\ptt =
    \po - \alpha \lambda \sum_{ i=0 }^{ t } 
    \bigl( \bfg(\bfx_{i}) + \lambda \bfp_{i} \bigr)$.
Using $ \po = \zero $ and $ \alpha \leq 1 / L(\xo,\lambda) $, \cref{corollary: nlls} implies
\begin{align*}
    \| \ptt \|
    &=
    \alpha \lambda \bigl\| \gxo \bigr\|
    \sum_{ i=0 }^{ t } \biggl( 1 - \frac{\lambda^{2}}{L(\xo,\lambda)} \biggr)^{ i / 2 }
    \\ &\leq
    \frac{ \alpha \lambda }{ 1 - \sqrt{ 1 - \lambda^{2} / L(\xo,\lambda) } }
    \bigl\| \gxo \bigr\| 
    \\ &=
    \alpha L(\xo,\lambda) 
    \Bigl( 1 + \sqrt{ 1 - \lambda / L(\xo,\lambda) } \Bigr)
    \bigl\| \gxo \bigr\|/\lambda
    \\ &\leq
    \Bigl( 1 + \sqrt{ 1 - \lambda / L(\xo,\lambda) } \Bigr) \bigl\| \gxo \bigr\|/\lambda.
\end{align*}
Now, it follows that
\begin{align}\label{eq: regularization bias}
    \bigl\| \gxt \bigr\|
    &\leq
    \sqrt{ 1 - \lambda / L(\xo,\lambda) }^{ \; t } \bigl\| \gxo \bigr\|
    +
    \lambda \| \pt \|
    \nonumber \\ &\leq
    \sqrt{ 1 - \lambda / L(\xo,\lambda) }^{ \; t } \bigl\| \gxo \bigr\|
    + 
    \Bigl( 1 + \sqrt{ 1 - \lambda / L(\xo,\lambda) } \Bigr)
    \bigl\| \gxo \bigr\|,
\end{align}
for all iterations $ t \geq 0 $. Although worst-case and pessimistic, this bound gives a qualitative, albeit very loose, guide to the effect of $ \lambda $ in training. It appears that  smaller values of $ \lambda $ amount to smaller regularization bias, at the cost of slower convergence for GD, and vice-versa. This observation is also supported by our numerical simulations; see \cref{figure: sigmoid}. This, in
some very loose sense, is reminiscent of the classical bias-variance trade-off from statistics. Further invesigation of the role of $ \lambda $ in this context is left for future work.

\section{Experiments}

\begin{figure*}[h!tb]
    \centering
    \includegraphics{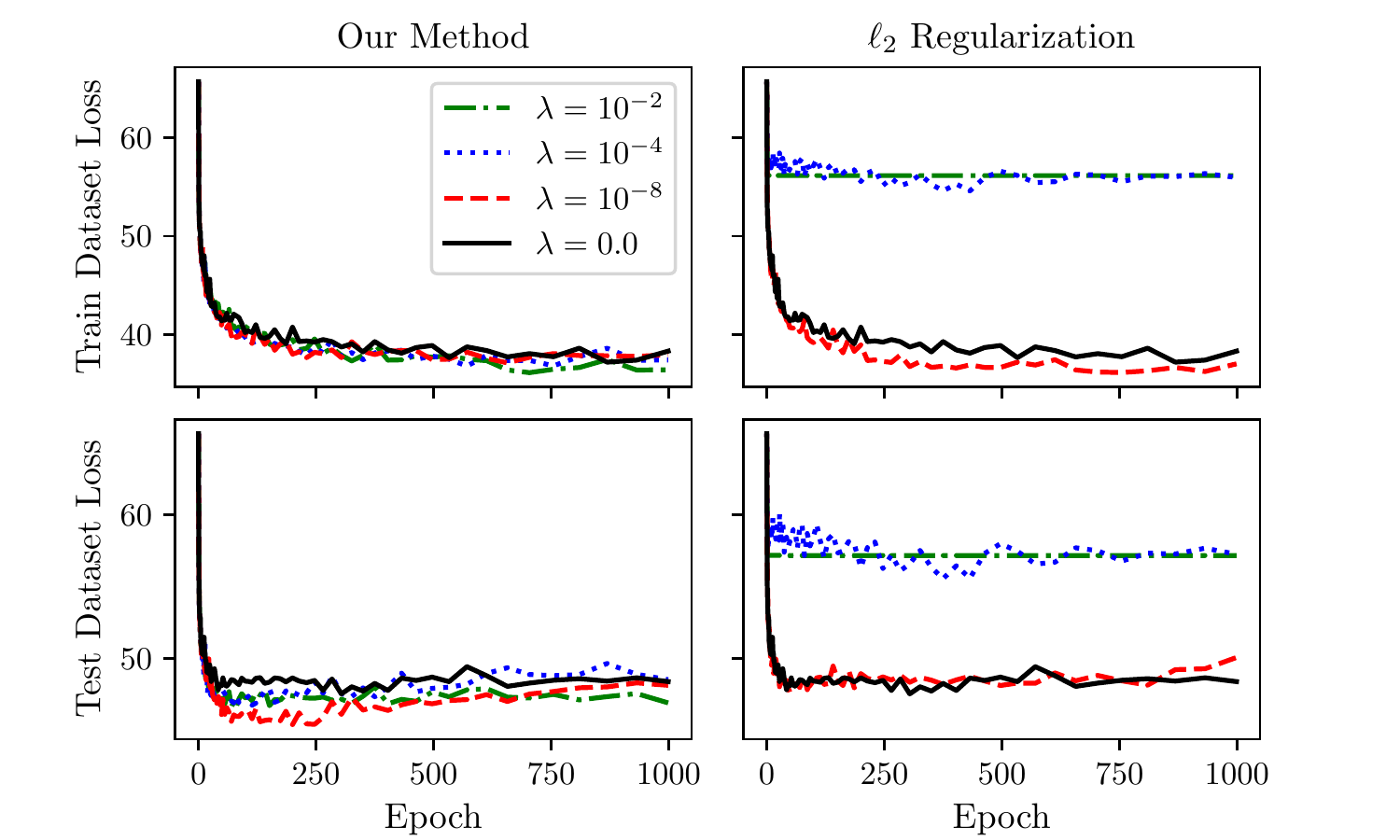}
    \caption{%
    Results of training a VAE model 
    on $ 16\times16 $ real-world handwritten digits from \cite{buscema1998metanet}.
    Here, we plot the non-regularized function value on the seen training dataset and unseen test dataset.
    Our method performs more consistently across all values of $\lambda$, 
    and with $ \lambda = 0.01 $ it yields the best non-regularized function values.
    This includes the run where $ \lambda = 0 $, i.e., 
    where the non-regularized function \eqref{eq: fx} is optimized directly.
    }
    \label{figure: VAE}
\end{figure*}
\begin{table*}[h!tb]
    \caption{Performance of trained VAE models after 1000 epochs. This is measured by the non-regularized function value on the seen training dataset and unseen test dataset, and the FID and KID scores using model's output images. Lower values are better, with the best value presented in boldface. Our method with $\lambda=10^{-2}$ attained the best value in each metric.}
    \label{table: vae}
    \vspace{1em}
    \centering
    \begin{tabular}{lccccccc}
        \toprule
        & 
        & \multicolumn{3}{c}{Our Method}
        & \multicolumn{3}{c}{$\ell_2$ Regularization}
        \\
        \cmidrule(lr){3-5}
        \cmidrule(lr){6-8}
        & $\lambda=0$
        & $\lambda=10^{-2}$
        & $\lambda=10^{-4}$
        & $\lambda=10^{-8}$
        & $\lambda=10^{-2}$
        & $\lambda=10^{-4}$
        & $\lambda=10^{-8}$
        \\
        \midrule
        Train Data Loss
        & $ 38.31 $
        & $ \mathbf{36.42} $
        & $ 37.40 $
        & $ 37.94 $
        & $ 56.15 $
        & $ 55.96 $
        & $ 37.02 $
        \\
        Test Data Loss
        & $ 48.42 $
        & $ \mathbf{46.94} $
        & $ 48.55 $
        & $ 48.16 $
        & $ 57.15 $
        & $ 57.24 $
        & $ 50.09 $
        \\
        FID Score
        & $ 402.8 $
        & $ \mathbf{376.0} $
        & $ 391.7 $
        & $ 390.2 $
        & $ 414.9 $
        & $ 410.2 $
        & $ 380.3 $
        \\
        KID Score
        & $ 0.534 $
        & $ \mathbf{0.478} $
        & $ 0.520 $
        & $ 0.512 $
        & $ 0.577 $
        & $ 0.627 $
        & $ 0.504 $
        \\
        \bottomrule
    \end{tabular}
\end{table*}
\begin{figure}[h!tb]
    \centering
    \includegraphics{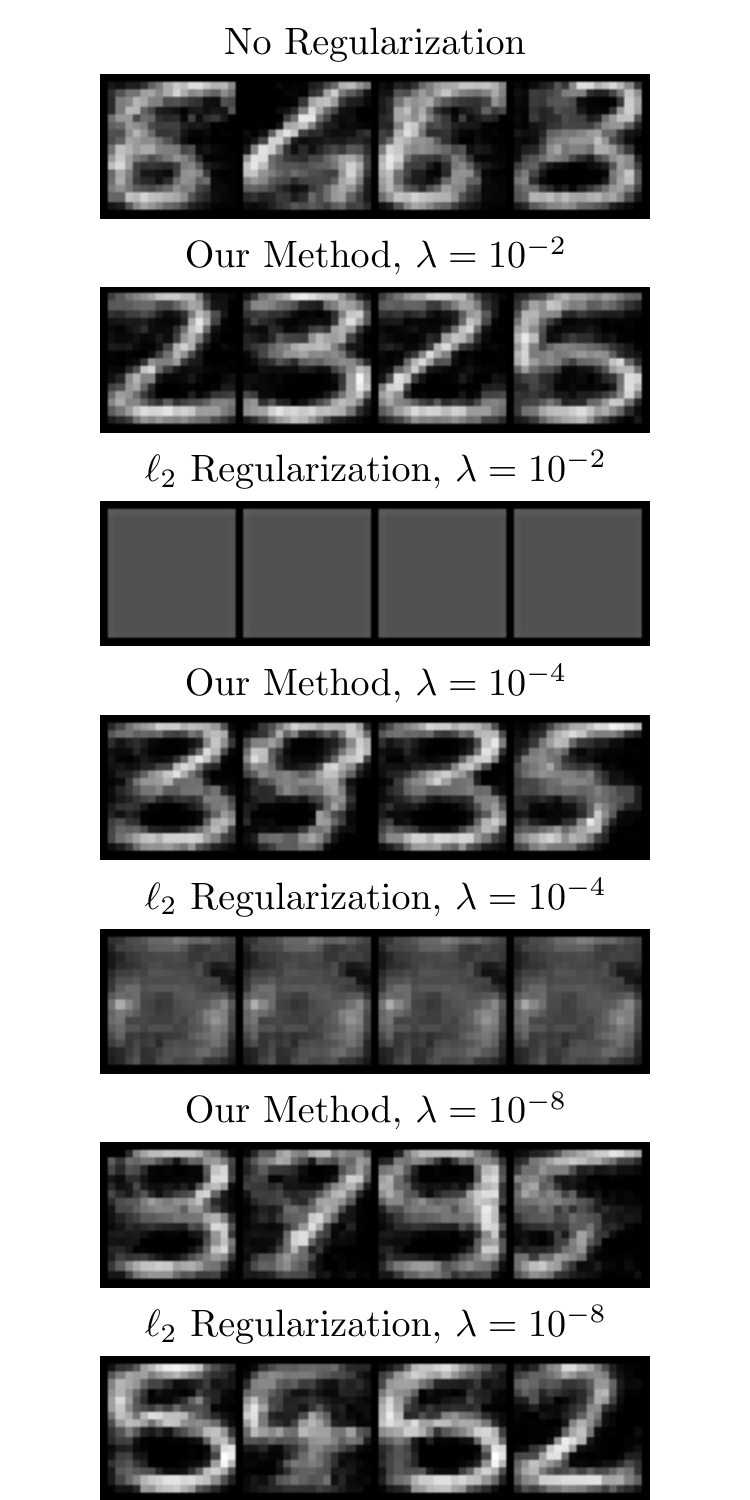}
    \caption{%
    VAE generated images after 1000 epochs.
    }
    \label{figure: VAE digits}
\end{figure}

\begin{figure*}[tb]
    \centering
    \includegraphics[scale=1]{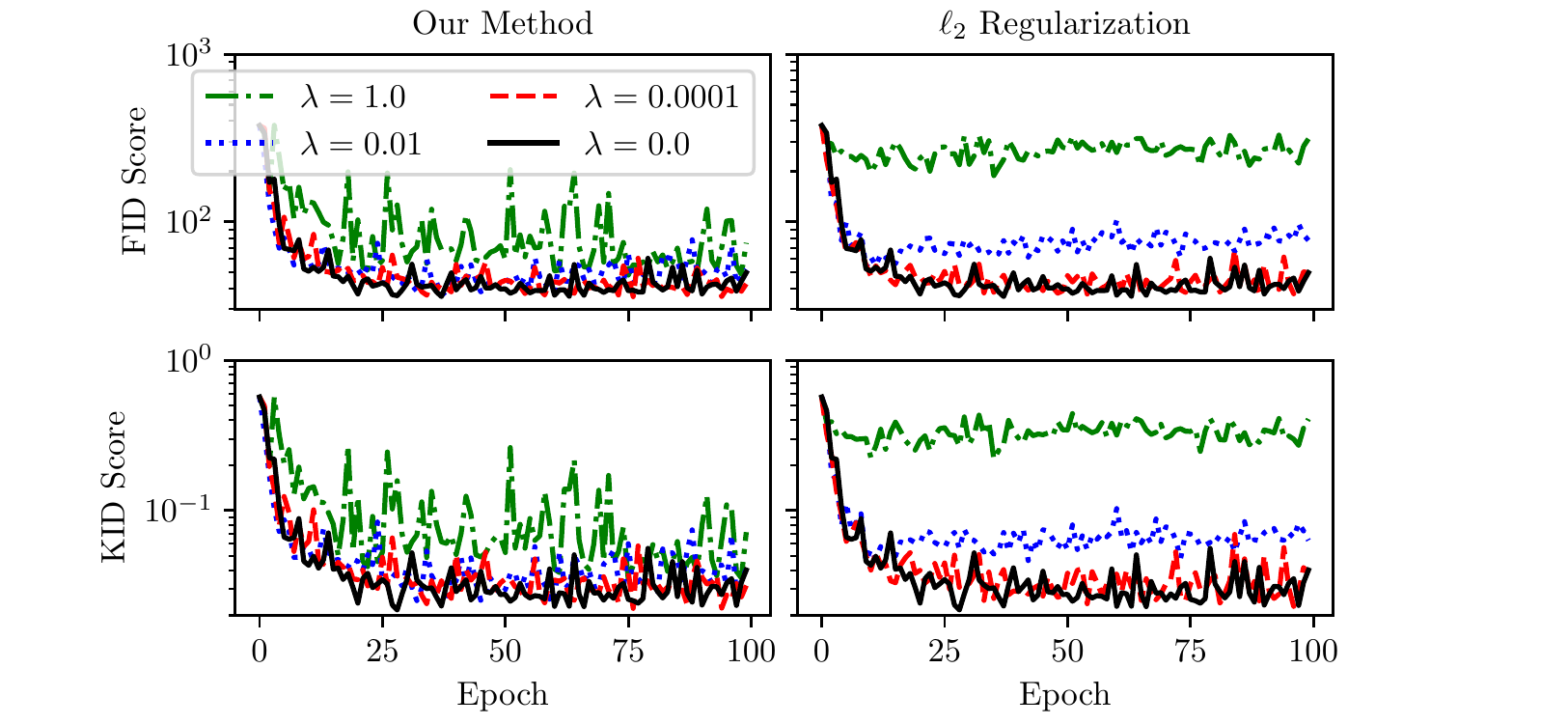}
    \caption{%
    Training a DCGAN on $ 64 \times 64 $ images 
    from the LFW dataset.
    Our method performs more consistently across all values of $\lambda$, and with $ \lambda = 0.0001 $ it yields the best 
    FID and KID scores amongst all runs.
    This includes the run where $ \lambda = 0 $, i.e., where \eqref{eq: fx}
    is optimized directly.
    Note that the vertical axes are logarithmic.
    }
    \label{figure: dcgan}
\end{figure*}
\begin{figure}[h!tb]
    \centering
    \includegraphics[scale=1.1]{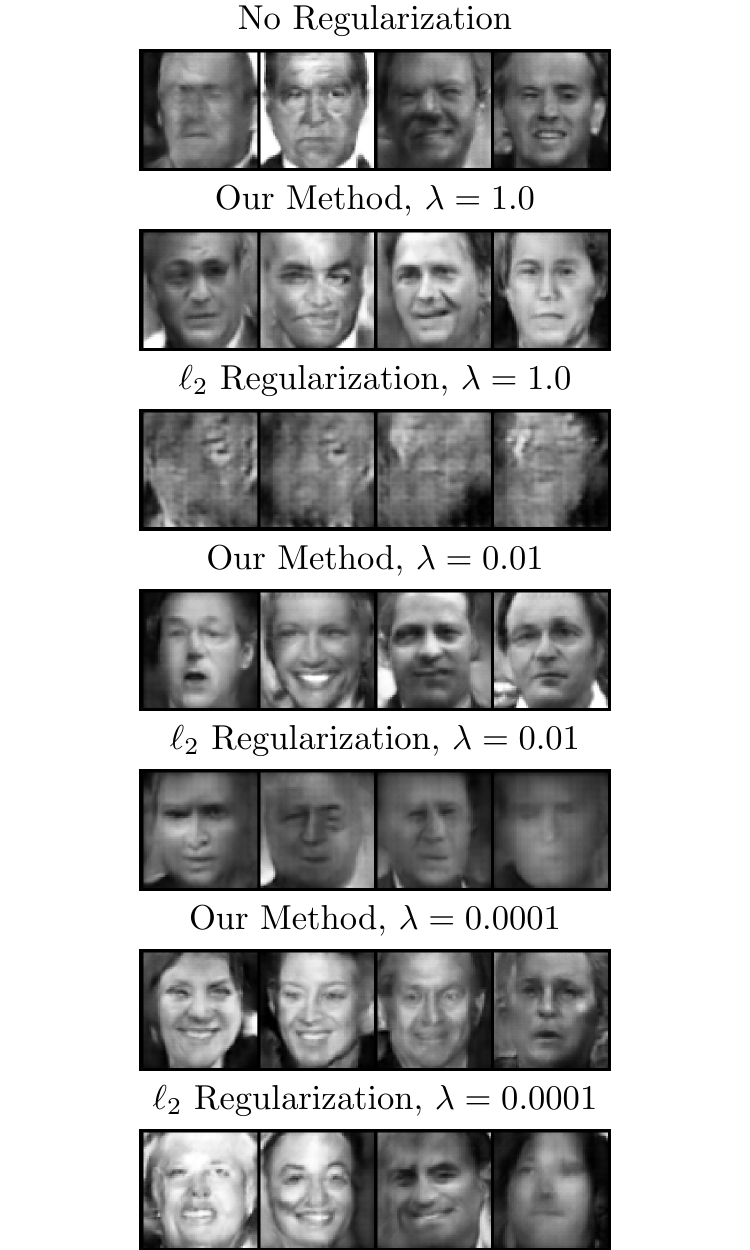}
    \caption{%
    DCGAN generated images after 100 epochs.
    }
    \label{figure: imgs}
\end{figure}

In this section, we examine the empirical performance of our proposed method in 
comparison to $ \ell_2 $-regularization.
To align with our theoretical results, we initialize $ \bfp $ with zeros in all 
experiments, i.e., $ \po = \zero $.
We will begin with, for illustrative purposes, a simple toy example of 
binary classification with squared-loss on synthetic data.
Afterwards, we consider more realistic and challenging examples of 
generative models.
All experiments are deterministic.
Namely, all models, in a given experiment, are initialized from the same weights.
Moreover, generated images are from the same random latent vector across 
regularization techniques.

For the generative models, we compute their
Fréchet Inception Distance (FID) score \citep{heusel2017gans}
and Kernel Inception Distance (KID) score \citep{binkowski2018demystifying}.
These are commonly used and well regarded metrics for assessing generative 
models.
We compute the score between the set of unseen real images from the test dataset and an equally sized set of generated fake images.
Note, these fake images are generated from the same latent vectors across 
the training epochs and regularization techniques.
Lower FID and KID scores are better.

\subsubsection*{Binary Classification With Squared-loss}

In \cref{figure: sigmoid}, and to verify our theory, we consider \eqref{eq: fx} with $ \fx = \|\sigma(\bfA\bfx)-\bfb\|^{2} / 2 $, 
where $\sigma$ denotes the element-wise sigmoid function, 
$ \bfA \in \real^{64\times 8} $, and $ \bfb $ has elements from $ \{0,1\} $.
The data $ \bfA $ and $ \bfb $ are synthetic.
GD is used with the same step-size across all runs.
Here, we observe that smaller values of $\lambda$ in \eqref{eq: fhxp} 
result in convergence to points that return smaller function values of 
the original non-regularized objective function 
\eqref{eq: fx}, i.e., a smaller regularization bias.
However, this comes at the cost of slower convergence rate. 
Indeed, increasing $\lambda$ amounts to a larger PL constant, 
which in turn results in faster convergence for GD on \eqref{eq: fhxp}.
The interpolation property of \cref{eq: fhxp}, to numerical accuracy, can be seen for all values of $\lambda$.

\subsubsection*{Variational Auto-encoder}

In Figures~\ref{figure: VAE} and \ref{figure: VAE digits}, 
we train a deep variational auto-encoder (VAE) model \citep{kingma2013auto} 
on $ 16\times16 $ real-world handwritten digits from \cite{buscema1998metanet}.
For this class of models, we use the commonly implemented training procedure 
that was introduced by \citet{radford2015unsupervised}.
Namely, to use the Adam optimizer \citep{kingma2014adam} 
with the momentum term $ \beta_1 = 0.5 $.
We also use a learning rate of $ 0.001 $ and batch size of $ 1 $.
Relative to $ \ell_2 $-regularization in Figure~\ref{figure: VAE}, 
our method is far more consistent between different values of $\lambda$.
This, yet again, offers a highly desirable practical advantage.
Indeed, the reduced difficulty of hyper-parameter tuning may amount to considerable cost savings over time in a practical setting.
As indicated by Table~\ref{table: vae}, our method obtains the model with the 
lowest non-regularized function value on the seen training dataset 
and unseen test dataset, as well as the best final FID and KID scores.
Some output images are shown in Figure~\ref{figure: VAE digits}.
Our approach yields sharper and more realistic images with less 
artifacts, compared with the alternative.

\subsubsection*{Deep Convolutional GAN}

Finally, in \cref{figure: dcgan,figure: imgs}, 
we consider the challenging problem of training a 
deep convolutional generative adversarial network (DCGAN) 
\citep{radford2015unsupervised} on $ 64 \times 64 $ real-world images 
from the \texttt{Labeled Faces in the Wild} (LFW) dataset \citep{LFWTech}.
For this class of models, we use the commonly implemented training procedure 
that was introduced by \citet{radford2015unsupervised}.
Namely, to use the Adam optimizer \citep{kingma2014adam} 
with the momentum term $ \beta_1 = 0.5 $.
We also use a learning rate of $ 0.0001 $ and batch size of $ 1 $.
As was the case of the VAE, our method is far more consistent 
than $ \ell_2 $-regularization between different values of $ \lambda $.
Moreover, our approach yields more realistic and consistent images, with less 
artifacts, compared with the alternative in Figure~\ref{figure: imgs}.

\section*{Acknowledgements}
Both authors gratefully acknowledge the generous support by the Australian Research Council Centre of Excellence for Mathematical \& Statistical Frontiers (ACEMS). Fred Roosta was partially supported by the ARC DECRA Award (DE180100923). This material is based on research partially sponsored by DARPA and the Air Force Research Laboratory under agreement number FA8750-17-2-0122. The U.S. Government is authorized to reproduce and distribute reprints for Governmental purposes notwithstanding any copyright notation thereon. The views and conclusions contained herein are those of the authors and should not be interpreted as necessarily representing the official policies or endorsements, either expressed or implied, of DARPA and the Air Force Research Laboratory or the U.S. Government.

\bibliographystyle{plainnat}
\bibliography{bibliography}

\begin{thebibliography}{54}
\providecommand{\natexlab}[1]{#1}
\providecommand{\url}[1]{\texttt{#1}}
\expandafter\ifx\csname urlstyle\endcsname\relax
  \providecommand{\doi}[1]{doi: #1}\else
  \providecommand{\doi}{doi: \begingroup \urlstyle{rm}\Url}\fi

\bibitem[Ajalloeian and Stich(2020)]{ajalloeian2020analysis}
Ahmad Ajalloeian and Sebastian~U Stich.
\newblock {Analysis of SGD with biased gradient estimators}.
\newblock \emph{arXiv preprint arXiv:2008.00051}, 2020.

\bibitem[Anil et~al.(2020)Anil, Gupta, Koren, Regan, and
  Singer]{anil2020second}
Rohan Anil, Vineet Gupta, Tomer Koren, Kevin Regan, and Yoram Singer.
\newblock {Second Order Optimization Made Practical}.
\newblock \emph{arXiv preprint arXiv:2002.09018}, 2020.

\bibitem[Baldi and Sadowski(2013)]{baldi2013understanding}
Pierre Baldi and Peter~J Sadowski.
\newblock Understanding dropout.
\newblock \emph{Advances in neural information processing systems},
  26:\penalty0 2814--2822, 2013.

\bibitem[Bassily et~al.(2018)Bassily, Belkin, and Ma]{bassily2018exponential}
Raef Bassily, Mikhail Belkin, and Siyuan Ma.
\newblock {On exponential convergence of SGD in non-convex over-parametrized
  learning}.
\newblock \emph{arXiv preprint arXiv:1811.02564}, 2018.

\bibitem[Basu et~al.(2019)Basu, Data, Karakus, and
  Diggavi]{NEURIPS2019_d202ed5b}
Debraj Basu, Deepesh Data, Can Karakus, and Suhas Diggavi.
\newblock Qsparse-local-{SGD}: distributed {SGD} with quantization,
  sparsification and local computations.
\newblock In H.~Wallach, H.~Larochelle, A.~Beygelzimer, F.~d\textquotesingle
  Alch\'{e}-Buc, E.~Fox, and R.~Garnett, editors, \emph{Advances in Neural
  Information Processing Systems}, volume~32, pages 14695--14706. Curran
  Associates, Inc., 2019.

\bibitem[Bates and Watts(2007)]{bates2007nonlinear}
D.M. Bates and D.G. Watts.
\newblock \emph{Nonlinear {R}egression {A}nalysis and {I}ts {A}pplications}.
\newblock Wiley Series in Probability and Statistics. Wiley, 2007.

\bibitem[Bellavia et~al.(2020)Bellavia, Gurioli, and
  Morini]{bellavia2020adaptive}
Stefania Bellavia, Gianmarco Gurioli, and Benedetta Morini.
\newblock {Adaptive cubic regularization methods with dynamic inexact Hessian
  information and applications to finite-sum minimization}.
\newblock \emph{IMA Journal of Numerical Analysis}, 2020.

\bibitem[Betancourt(2017)]{betancourt2017conceptual}
Michael Betancourt.
\newblock {A conceptual introduction to Hamiltonian Monte Carlo}.
\newblock \emph{arXiv preprint arXiv:1701.02434}, 2017.

\bibitem[Bi{\'n}kowski et~al.(2018)Bi{\'n}kowski, Sutherland, Arbel, and
  Gretton]{binkowski2018demystifying}
Miko{\l}aj Bi{\'n}kowski, Danica~J Sutherland, Michael Arbel, and Arthur
  Gretton.
\newblock Demystifying mmd gans.
\newblock \emph{arXiv preprint arXiv:1801.01401}, 2018.

\bibitem[Blanchet et~al.(2019)Blanchet, Cartis, Menickelly, and
  Scheinberg]{blanchet2019convergence}
Jose Blanchet, Coralia Cartis, Matt Menickelly, and Katya Scheinberg.
\newblock Convergence rate analysis of a stochastic trust-region method via
  supermartingales.
\newblock \emph{INFORMS journal on optimization}, 1\penalty0 (2):\penalty0
  92--119, 2019.

\bibitem[Buscema(1998)]{buscema1998metanet}
Massimo Buscema.
\newblock {MetaNet*}: The theory of independent judges.
\newblock \emph{Substance use \& misuse}, 33\penalty0 (2):\penalty0 439--461,
  1998.

\bibitem[Cambini and Martein(2008)]{cambini2008generalized}
A.~Cambini and L.~Martein.
\newblock \emph{{Generalized Convexity and Optimization: Theory and
  Applications}}.
\newblock Lecture Notes in Economics and Mathematical Systems. Springer Berlin
  Heidelberg, 2008.

\bibitem[Chen et~al.(2019)Chen, Yin, Fisher, Chaudhuri, and
  Zhang]{Chen_2019_ICCV}
Zhiqin Chen, Kangxue Yin, Matthew Fisher, Siddhartha Chaudhuri, and Hao Zhang.
\newblock {BAE-NET}: branched autoencoder for shape co-segmentation.
\newblock In \emph{Proceedings of the IEEE/CVF International Conference on
  Computer Vision (ICCV)}, October 2019.

\bibitem[Cutkosky and Orabona(2019)]{NEURIPS2019_b8002139}
Ashok Cutkosky and Francesco Orabona.
\newblock Momentum-based variance reduction in non-convex {SGD}.
\newblock In H.~Wallach, H.~Larochelle, A.~Beygelzimer, F.~d\textquotesingle
  Alch\'{e}-Buc, E.~Fox, and R.~Garnett, editors, \emph{Advances in Neural
  Information Processing Systems}, volume~32, pages 15236--15245. Curran
  Associates, Inc., 2019.

\bibitem[Golatkar et~al.(2019)Golatkar, Achille, and Soatto]{golatkar2019time}
Aditya~Sharad Golatkar, Alessandro Achille, and Stefano Soatto.
\newblock Time matters in regularizing deep networks: weight decay and data
  augmentation affect early learning dynamics, matter little near convergence.
\newblock In \emph{Advances in Neural Information Processing Systems}, pages
  10678--10688, 2019.

\bibitem[Goodfellow et~al.(2016)Goodfellow, Bengio, and
  Courville]{goodfellow2016deep}
Ian Goodfellow, Yoshua Bengio, and Aaron Courville.
\newblock \emph{{Deep Learning}}.
\newblock MIT press, 2016.

\bibitem[Gower et~al.(2021)Gower, Sebbouh, and Loizou]{gower2021sgd}
Robert Gower, Othmane Sebbouh, and Nicolas Loizou.
\newblock {SGD for structured nonconvex functions: learning rates, minibatching
  and interpolation}.
\newblock In \emph{International Conference on Artificial Intelligence and
  Statistics}, pages 1315--1323. PMLR, 2021.

\bibitem[Gupta et~al.(2018)Gupta, Koren, and Singer]{gupta2018shampoo}
Vineet Gupta, Tomer Koren, and Yoram Singer.
\newblock {Shampoo: preconditioned stochastic tensor optimization}.
\newblock \emph{arXiv preprint arXiv:1802.09568}, 2018.

\bibitem[Haddadpour et~al.(2019)Haddadpour, Kamani, Mahdavi, and
  Cadambe]{pmlr-v97-haddadpour19a}
Farzin Haddadpour, Mohammad~Mahdi Kamani, Mehrdad Mahdavi, and Viveck Cadambe.
\newblock Trading redundancy for communication: speeding up distributed {SGD}
  for non-convex optimization.
\newblock In Kamalika Chaudhuri and Ruslan Salakhutdinov, editors,
  \emph{Proceedings of the 36th International Conference on Machine Learning},
  volume~97 of \emph{Proceedings of Machine Learning Research}, pages
  2545--2554, Long Beach, California, USA, 09--15 Jun 2019. PMLR.

\bibitem[Heusel et~al.(2017)Heusel, Ramsauer, Unterthiner, Nessler, and
  Hochreiter]{heusel2017gans}
Martin Heusel, Hubert Ramsauer, Thomas Unterthiner, Bernhard Nessler, and Sepp
  Hochreiter.
\newblock {GAN}s trained by a two time-scale update rule converge to a local
  {N}ash equilibrium.
\newblock \emph{Advances in neural information processing systems}, 30, 2017.

\bibitem[Hinton et~al.(2012)Hinton, Srivastava, Krizhevsky, Sutskever, and
  Salakhutdinov]{hinton2012improving}
Geoffrey~E Hinton, Nitish Srivastava, Alex Krizhevsky, Ilya Sutskever, and
  Ruslan~R Salakhutdinov.
\newblock Improving neural networks by preventing co-adaptation of feature
  detectors.
\newblock \emph{arXiv preprint arXiv:1207.0580}, 2012.

\bibitem[Huang et~al.(2007)Huang, Ramesh, Berg, and Learned-Miller]{LFWTech}
Gary~B. Huang, Manu Ramesh, Tamara Berg, and Erik Learned-Miller.
\newblock Labeled faces in the wild: a database for studying face recognition
  in unconstrained environments.
\newblock Technical Report 07-49, University of Massachusetts, Amherst, October
  2007.

\bibitem[Karimi et~al.(2016)Karimi, Nutini, and
  Schmidt]{10.1007/978-3-319-46128-1_50}
Hamed Karimi, Julie Nutini, and Mark Schmidt.
\newblock Linear convergence of gradient and proximal-gradient methods under
  the {Polyak-{\L}ojasiewicz} condition.
\newblock In Paolo Frasconi, Niels Landwehr, Giuseppe Manco, and Jilles
  Vreeken, editors, \emph{Machine Learning and Knowledge Discovery in
  Databases}, pages 795--811, Cham, 2016. Springer International Publishing.
\newblock ISBN 978-3-319-46128-1.

\bibitem[Kawaguchi and Kaelbling(2020)]{kawaguchi2020elimination}
Kenji Kawaguchi and Leslie Kaelbling.
\newblock Elimination of all bad local minima in deep learning.
\newblock In \emph{International Conference on Artificial Intelligence and
  Statistics}, pages 853--863. PMLR, 2020.

\bibitem[Kingma and Ba(2014)]{kingma2014adam}
Diederik Kingma and Jimmy Ba.
\newblock Adam: A method for stochastic optimization.
\newblock \emph{arXiv preprint arXiv:1412.6980}, 2014.

\bibitem[Kingma and Welling(2013)]{kingma2013auto}
Diederik~P Kingma and Max Welling.
\newblock Auto-encoding variational bayes.
\newblock \emph{arXiv preprint arXiv:1312.6114}, 2013.

\bibitem[Krogh and Hertz(1992)]{krogh1992simple}
Anders Krogh and John~A Hertz.
\newblock A simple weight decay can improve generalization.
\newblock In \emph{Advances in neural information processing systems}, pages
  950--957, 1992.

\bibitem[Li et~al.(2018)Li, Xu, Taylor, Studer, and
  Goldstein]{li2018visualizing}
Hao Li, Zheng Xu, Gavin Taylor, Christoph Studer, and Tom Goldstein.
\newblock Visualizing the loss landscape of neural nets.
\newblock In \emph{Advances in Neural Information Processing Systems}, pages
  6389--6399, 2018.

\bibitem[Liang et~al.(2018)Liang, Sun, Lee, and Srikant]{liang2018adding}
Shiyu Liang, Ruoyu Sun, Jason~D Lee, and Rayadurgam Srikant.
\newblock Adding one neuron can eliminate all bad local minima.
\newblock \emph{arXiv preprint arXiv:1805.08671}, 2018.

\bibitem[Liu et~al.(2020)Liu, Zhu, and Belkin]{liu2020toward}
Chaoyue Liu, Libin Zhu, and Mikhail Belkin.
\newblock Toward a theory of optimization for over-parameterized systems of
  non-linear equations: the lessons of deep learning.
\newblock \emph{arXiv preprint arXiv:2003.00307}, 2020.

\bibitem[Liu and Roosta(2021)]{liu2019stability}
Yang Liu and Fred Roosta.
\newblock {Convergence of Newton-MR under Inexact Hessian Information}.
\newblock \emph{SIAM Journal on Optimization}, 31\penalty0 (1):\penalty0
  59--90, 2021.

\bibitem[Loshchilov and Hutter(2017)]{loshchilov2017decoupled}
Ilya Loshchilov and Frank Hutter.
\newblock Decoupled weight decay regularization.
\newblock \emph{arXiv preprint arXiv:1711.05101}, 2017.

\bibitem[Mao et~al.(2017)Mao, Li, Xie, Lau, Wang, and
  Paul~Smolley]{Mao_2017_ICCV}
Xudong Mao, Qing Li, Haoran Xie, Raymond~Y.K. Lau, Zhen Wang, and Stephen
  Paul~Smolley.
\newblock Least squares generative adversarial networks.
\newblock In \emph{Proceedings of the IEEE International Conference on Computer
  Vision (ICCV)}, October 2017.

\bibitem[Mishra and Giorgi(2008)]{mishra2008invexity}
Shashi~K Mishra and Giorgio Giorgi.
\newblock \emph{{Invexity and Optimization}}, volume~88.
\newblock Springer Science \& Business Media, 2008.

\bibitem[Mohri et~al.(2018)Mohri, Rostamizadeh, and
  Talwalkar]{mohri2018foundations}
Mehryar Mohri, Afshin Rostamizadeh, and Ameet Talwalkar.
\newblock \emph{Foundations of machine learning}.
\newblock MIT press, 2018.

\bibitem[Muecke et~al.(2019)Muecke, Neu, and Rosasco]{NEURIPS2019_4d0b954f}
Nicole Muecke, Gergely Neu, and Lorenzo Rosasco.
\newblock Beating {SGD} saturation with tail-averaging and minibatching.
\newblock In H.~Wallach, H.~Larochelle, A.~Beygelzimer, F.~d\textquotesingle
  Alch\'{e}-Buc, E.~Fox, and R.~Garnett, editors, \emph{Advances in Neural
  Information Processing Systems}, volume~32, pages 12568--12577. Curran
  Associates, Inc., 2019.

\bibitem[Neal et~al.(2011)]{neal2011mcmc}
Radford~M Neal et~al.
\newblock {MCMC using Hamiltonian dynamics}.
\newblock \emph{Handbook of markov chain monte carlo}, 2\penalty0
  (11):\penalty0 2, 2011.

\bibitem[Oymak and Soltanolkotabi(2019)]{oymak2019overparameterized}
Samet Oymak and Mahdi Soltanolkotabi.
\newblock Overparameterized nonlinear learning: gradient descent takes the
  shortest path?
\newblock In \emph{International Conference on Machine Learning}, pages
  4951--4960, 2019.

\bibitem[Polyak(1963)]{Polyak1963}
Boris Polyak.
\newblock Gradient methods for the minimisation of functionals.
\newblock \emph{USSR Computational Mathematics and Mathematical Physics},
  3:\penalty0 864--878, 12 1963.
\newblock \doi{10.1016/0041-5553(63)90382-3}.

\bibitem[Radford et~al.(2015)Radford, Metz, and
  Chintala]{radford2015unsupervised}
Alec Radford, Luke Metz, and Soumith Chintala.
\newblock Unsupervised representation learning with deep convolutional
  generative adversarial networks.
\newblock \emph{arXiv preprint arXiv:1511.06434}, 2015.

\bibitem[Roosta et~al.(2014{\natexlab{a}})Roosta, van~den Doel, and
  Ascher]{rodoas1}
Fred Roosta, Kees van~den Doel, and Uri Ascher.
\newblock {Stochastic algorithms for inverse problems involving PDEs and many
  measurements}.
\newblock \emph{SIAM J. Scientific Computing}, 36\penalty0 (5):\penalty0
  S3--S22, 2014{\natexlab{a}}.

\bibitem[Roosta et~al.(2014{\natexlab{b}})Roosta, van~den Doel, and
  Ascher]{rodoas2}
Fred Roosta, Kees van~den Doel, and Uri Ascher.
\newblock {Data completion and stochastic algorithms for PDE inversion problems
  with many measurements}.
\newblock \emph{Electronic Transactions on Numerical Analysis}, 42:\penalty0
  177--196, 2014{\natexlab{b}}.

\bibitem[Shalev-Shwartz and Ben-David(2014)]{shalev2014understanding}
Shai Shalev-Shwartz and Shai Ben-David.
\newblock \emph{{Understanding machine learning: From theory to algorithms}}.
\newblock Cambridge university press, 2014.

\bibitem[Shorten and Khoshgoftaar(2019)]{shorten2019survey}
Connor Shorten and Taghi~M Khoshgoftaar.
\newblock A survey on image data augmentation for deep learning.
\newblock \emph{Journal of Big Data}, 6\penalty0 (1):\penalty0 1--48, 2019.

\bibitem[Soudry et~al.(2018)Soudry, Hoffer, Nacson, Gunasekar, and
  Srebro]{JMLR:v19:18-188}
Daniel Soudry, Elad Hoffer, Mor~Shpigel Nacson, Suriya Gunasekar, and Nathan
  Srebro.
\newblock The implicit bias of gradient descent on separable data.
\newblock \emph{Journal of Machine Learning Research}, 19\penalty0
  (70):\penalty0 1--57, 2018.
\newblock URL \url{http://jmlr.org/papers/v19/18-188.html}.

\bibitem[Tripuraneni et~al.(2018)Tripuraneni, Stern, Jin, Regier, and
  Jordan]{tripuraneni2017stochasticcubic}
Nilesh Tripuraneni, Mitchell Stern, Chi Jin, Jeffrey Regier, and Michael~I
  Jordan.
\newblock {Stochastic cubic regularization for fast nonconvex optimization}.
\newblock In \emph{Advances in neural information processing systems}, pages
  2899--2908, 2018.

\bibitem[Vaswani et~al.(2019)Vaswani, Bach, and Schmidt]{pmlr-v89-vaswani19a}
Sharan Vaswani, Francis Bach, and Mark Schmidt.
\newblock Fast and faster convergence of {SGD} for over-parameterized models
  and an accelerated perceptron.
\newblock In Kamalika Chaudhuri and Masashi Sugiyama, editors,
  \emph{Proceedings of Machine Learning Research}, volume~89 of
  \emph{Proceedings of Machine Learning Research}, pages 1195--1204. PMLR,
  16--18 Apr 2019.

\bibitem[Vogels et~al.(2019)Vogels, Karimireddy, and
  Jaggi]{NEURIPS2019_d9fbed9d}
Thijs Vogels, Sai~Praneeth Karimireddy, and Martin Jaggi.
\newblock {PowerSGD}: practical low-rank gradient compression for distributed
  optimization.
\newblock In H.~Wallach, H.~Larochelle, A.~Beygelzimer, F.~d\textquotesingle
  Alch\'{e}-Buc, E.~Fox, and R.~Garnett, editors, \emph{Advances in Neural
  Information Processing Systems}, volume~32, pages 14259--14268. Curran
  Associates, Inc., 2019.

\bibitem[Wang et~al.(2019)Wang, Zhou, Liang, and Lan]{wang2019stochastic}
Zhe Wang, Yi~Zhou, Yingbin Liang, and Guanghui Lan.
\newblock Stochastic variance-reduced cubic regularization for nonconvex
  optimization.
\newblock In \emph{The 22nd International Conference on Artificial Intelligence
  and Statistics}, pages 2731--2740. PMLR, 2019.

\bibitem[Xu et~al.(2020)Xu, Roosta, and Mahoney]{xuNonconvexTheoretical2017}
Peng Xu, Fred Roosta, and Michael~W Mahoney.
\newblock {Newton-type methods for non-convex optimization under inexact
  Hessian information}.
\newblock \emph{Mathematical Programming}, 184\penalty0 (1):\penalty0 35--70,
  2020.

\bibitem[Yao et~al.(2020)Yao, Xu, Roosta, and Mahoney]{yao2018inexact}
Zhewei Yao, Peng Xu, Fred Roosta, and Michael~W Mahoney.
\newblock {Inexact non-convex Newton-type methods}.
\newblock \emph{INFORMS Journal on Optimization}, 2020.
\newblock doi.org/10.1287/ijoo.2019.0043.

\bibitem[Yu and Jin(2019)]{pmlr-v97-yu19c}
Hao Yu and Rong Jin.
\newblock On the computation and communication complexity of parallel {SGD}
  with dynamic batch sizes for stochastic non-convex optimization.
\newblock In Kamalika Chaudhuri and Ruslan Salakhutdinov, editors,
  \emph{Proceedings of the 36th International Conference on Machine Learning},
  volume~97 of \emph{Proceedings of Machine Learning Research}, pages
  7174--7183, Long Beach, California, USA, 09--15 Jun 2019. PMLR.

\bibitem[Yuan et~al.(2018)Yuan, Yan, Jin, and Yang]{yuan2018stagewise}
Zhuoning Yuan, Yan Yan, Rong Jin, and Tianbao Yang.
\newblock {Stagewise training accelerates convergence of testing error over
  SGD}.
\newblock \emph{arXiv preprint arXiv:1812.03934}, 2018.

\bibitem[Zhang et~al.(2018)Zhang, Wang, Xu, and Grosse]{zhang2018three}
Guodong Zhang, Chaoqi Wang, Bowen Xu, and Roger Grosse.
\newblock Three mechanisms of weight decay regularization.
\newblock \emph{arXiv preprint arXiv:1810.12281}, 2018.

\end{thebibliography}

\end{document}